\documentclass[12pt]{article}

\usepackage[english]{babel}

\usepackage[letterpaper,top=2cm,bottom=2cm,left=3cm,right=2cm,marginparwidth=1.75cm]{geometry}
\usepackage{tikz}
\usepackage{amsmath,float,bm}
\usepackage{natbib}
\usepackage[colorlinks=true,
linkcolor=blue,
citecolor=blue,
urlcolor=blue]{hyperref}
\usepackage{pgfplots}
\usepgfplotslibrary{fillbetween}
\usepackage{amsmath,amssymb}
\pgfplotsset{compat=1.16}

\usepackage{titlesec}

\usepackage{amsmath,amsfonts,amsthm,mathrsfs,dsfont}
\usepackage{graphicx,color}
\usepackage{authblk}

\newtheorem{theorem}{Theorem}
\newtheoremstyle{normalremark}
  {3pt}   
  {3pt}   
  {\normalfont} 
  {}      
  {
  \bfseries} 
  {.}     
  {.5em}  
  {}      

\theoremstyle{normalremark}
\newtheorem{remark}{Remark}
\newtheorem{definition}{Definition}
\newtheorem{proposition}{Proposition}
\newtheorem{corollary}{Corollary}
\newtheorem{example}{Example}

\title{ \LARGE
\bf Unified formulas for conditional quantities and transportation functionals}

\author[1]{Roberto Vila\footnote{Corresponding author. Roberto Vila, e-mail: rovig161@gmail.com}}
\affil[1]{\it\small Department of Statistics, University of Brasilia, Brasilia, Brazil}
\author[1]{Eduardo Nakano}
\author[2]{Chang C. Y. Dorea}
\affil[2]{\it\small Department of Mathematics, University of Brasilia, Brasilia, Brazil}

\begin{document}
\maketitle

\begin{abstract}
This paper develops a unified probabilistic framework based on distributional derivatives and Dirac delta representations for the analysis of conditional and transportation-related quantities. General identities are established for arbitrary random variables, encompassing absolutely continuous, discrete, and mixed distributions. The proposed approach yields unified formulas for conditional expectations, conditional distributions, hazard functions, and improper distributions, revealing a common localization mechanism underlying these classical concepts.

The framework is further combined with copula methods to investigate transportation and dispersion functionals through dependence structures. Exploiting the extremal properties of the Fréchet--Hoeffding bounds together with expectation inequalities induced by $\Delta$-antitonic functions, sharp bounds are derived for absolute difference moments under fixed marginals. These results lead to concise derivations of quantile representations for the Wasserstein distance and a corresponding upper transportation functional, as well as survival-function representations and bounds for generalized absolute difference moments. As a particular case, new representations are obtained for the bivariate Gini mean difference and the associated bivariate Gini index.

Applications are given to Wasserstein-type functionals arising in the normal approximation of standardized counting distributions, including Poisson, Binomial, and Negative Binomial models, for which explicit quantile representations are derived. Overall, the results establish explicit links among conditional structures, dependence modeling, dispersion measures, normal approximation, and optimal transport, providing a unified perspective on several fundamental constructions in probability and mathematical statistics.

\end{abstract}


{
\hypersetup{linkcolor=black}
\tableofcontents
}

\section{Introduction}

Conditional expectations, conditional distributions, and transportation functionals play fundamental roles in probability theory and mathematical statistics. Conditional distributions constitute one of the basic tools of statistical inference, Bayesian modeling, stochastic processes, and uncertainty quantification, whereas hazard functions are central to survival analysis, reliability theory, and event-history modeling. Transportation functionals, in turn, play a prominent role in optimal transport theory, distributional approximation, and robust statistical methods. Closely related quantities, such as Gini-type measures of dispersion and inequality, can also be expressed through absolute difference moments and admit natural interpretations in terms of dependence and transportation. Although these concepts are often studied in distinct contexts, many of their classical representations are connected through a common probabilistic structure.

The use of generalized functions in probability theory has a long history. In particular, the Dirac delta distribution provides a convenient framework for describing densities, point masses, and conditioning operations in a unified manner; see, for instance, \cite{Gelfand1964,Kanwal1998}. Distributional methods have proved useful in probability, stochastic processes, and mathematical physics, where weak derivatives frequently provide natural extensions of classical analytical constructions. Moreover, generalized-function techniques provide a natural language for describing singular probability measures and localization phenomena, both of which arise naturally in conditioning and transportation problems.

Conditional expectations given events of probability zero constitute a classical topic in measure-theoretic probability. Standard treatments are based on regular conditional probabilities and disintegration theory; see \cite{Billingsley1995,Kallenberg2002,Durrett2019}. In the present work, we adopt an alternative viewpoint based on limiting conditional events and distributional derivatives. This approach yields unified formulas valid for absolutely continuous, discrete, and mixed random variables, thereby revealing a common structure underlying several classical probabilistic quantities.

A second objective of the paper is to investigate transportation-type functionals through dependence structures. The connection between dependence and extremal expectations has been extensively studied in copula theory; see \cite{Nelsen2006,Joe2014}. In particular, the Fr\'echet--Hoeffding bounds play a fundamental role in identifying extremal couplings and deriving sharp inequalities under fixed marginals \cite{Genest1999,Lux2017}. These ideas naturally connect with optimal transport and Wasserstein-type distances, which have become fundamental tools in probability theory, statistics, data science, and machine learning, with applications ranging from distributional approximation and goodness-of-fit assessment to distributionally robust optimization; see \cite{Villani2003,Villani2009}. They also connect with dependence-based measures of dispersion, such as the bivariate Gini mean difference and the associated bivariate Gini index introduced in \cite{Capaldo2025}.

Motivated by these developments, we combine distributional arguments with copula-based extremal principles to derive sharp bounds for moments of the form $\mathbb E|X-Y|^r$ under prescribed marginals. As consequences, we obtain concise derivations of the classical quantile representation of the Wasserstein distance and of a corresponding upper transportation functional. We further derive survival-function representations and bounds for generalized absolute difference moments, yielding, as a particular case, explicit bounds for the bivariate Gini mean difference and the associated bivariate Gini index. In addition, the resulting quantile representations provide explicit formulas for Wasserstein-type functionals arising in the normal approximation of standardized counting distributions, including Poisson, Binomial, and Negative Binomial models. More broadly, the results establish explicit links among conditional structures, dependence modeling, dispersion measures, normal approximation, and transportation problems within a unified probabilistic framework.

The paper is organized as follows. Section~\ref{Distributional derivatives} reviews basic facts concerning distributional derivatives and Dirac delta distributions. Section~\ref{Unified formulas for statistical quantities} develops unified formulas for conditional expectations, conditional laws, hazard functions, and improper distributions. Section~\ref{Wasserstein distance formulas} establishes copula-based bounds for transportation functionals, derives quantile and survival-function representations for Wasserstein-type and related functionals, and presents applications to bivariate Gini measures and to the normal approximation of classical counting distributions. Section~\ref{Concluding remarks} concludes the paper with a discussion of the main results and possible directions for future research.


\section{Distributional derivatives}\label{Distributional derivatives}

This section presents the distributional tools used throughout the paper, including weak derivatives, Dirac delta distributions, and identities involving the Fr\'echet-Hoeffding bounds.

Let \(g:\mathbb{R}\to\mathbb{R}\) be a locally integrable function, that is,
$
g\in L_{\mathrm{loc}}^1(\mathbb{R}),
$
meaning that
$
\int_K |g(t)|{\rm d}t<\infty,
$
for every compact set \(K\subset\mathbb{R}\).

The distributional derivative (or weak derivative) of \(g\) is the distribution \(g'\) defined by
\[
\int_{-\infty}^{\infty}
g'(x)\varphi(x){\rm d}x
=
-
\int_{-\infty}^{\infty}
g(x)\varphi'(x){\rm d}x,
\]
for every test function
$
\varphi\in C_c^\infty(\mathbb{R}),
$
where \(C_c^\infty(\mathbb{R})\) denotes the space of infinitely differentiable functions with compact support.


More generally, the distributional derivative of order $n\in\mathbb{N}$ of $g$ is the distribution $g^{(n)}$ defined by
\[
\int_{-\infty}^{\infty}
g^{(n)}(x)\varphi(x){\rm d}x
=
(-1)^n
\int_{-\infty}^{\infty}
g(x)\varphi^{(n)}(x){\rm d}x,
\]
for every
$
\varphi\in C_c^\infty(\mathbb{R}).
$

Moreover, if $g$ is classically $n$-times differentiable and
$
g^{(n)}\in L_{\mathrm{loc}}^1(\mathbb{R}),
$
then the distributional derivative of order $n$ coincides with the classical derivative. Indeed, by repeated integration by parts,
\[
(-1)^n
\int_{-\infty}^{\infty}
g(x)\varphi^{(n)}(x){\rm d}x
=
\int_{-\infty}^{\infty}
g^{(n)}(x)\varphi(x){\rm d}x,
\]
since all boundary terms vanish due to the compact support of $\varphi$.
\begin{example}[Heaviside {step} function]\label{ex-1}
	Let
	\[
	H_{x_0}(x)=\mathds{1}_{\{x\geqslant x_0\}}.
	\]
	
	Then $H_{x_0}$ is not differentiable at $x_0$ in the classical sense; however, its distributional derivative is the Dirac delta distribution located at $x=x_0$, namely,
	\[
	H_{x_0}'=\delta_{x_0}.
	\]
	
	Indeed,
	\[
	\int_{-\infty}^{\infty}
	\delta_{x_0}(x)\varphi(x){\rm d}x
	=
	\varphi(x_0)
	=
	-
	\int_{x_0}^\infty \varphi'(x){\rm d}x.
	\]
\end{example}

\begin{proposition}\label{pro-main}
	The following identity holds in the distributional sense:
	%
	%
	\begin{align*}
		\lim_{\varepsilon \to 0^+}
		{ 
			H_{x_0}(x)
			-
			H_{x_0}(x-\varepsilon)
			\over 
			\varepsilon
		}
		=
		H_{x_0}'(x)
		=
		\delta_{x_0}(x).
	\end{align*}
\end{proposition}

\begin{proof}
	For all $\varphi\in C_c^\infty(\mathbb{R})$, we must prove that
	\begin{align*}
		\int_{-\infty}^\infty 
		\lim_{\varepsilon\to 0^+}
		{ 
			H_{x_0}(x)
			-
			H_{x_0}(x-\varepsilon)
			\over 
			\varepsilon
		}
		\, \varphi(x)
		{\rm d} x
		=
		-
		\int_{-\infty}^\infty 
		H_{x_0}(x)
		\varphi'(x)
		{\rm d} x.
	\end{align*}
	
	Indeed, note that
	\begin{align}\label{eq-lim}
		\int_{-\infty}^\infty 
		\lim_{\varepsilon\to 0^+}
		{ 
			H_{x_0}(x)
			-
			H_{x_0}(x-\varepsilon)
			\over 
			\varepsilon
		}
		\, \varphi(x)
		{\rm d} x
		=
		\lim_{\varepsilon\to 0^+}
		{1\over \varepsilon}
		\int^{x_0+\varepsilon}_{x_0}
		\varphi(x)
		{\rm d} x,
	\end{align}
	where the interchange between the limit and the integral follows from the dominated convergence theorem.
	
	Using the mean value theorem for integrals, there exists
	$
	x_0 \leqslant \xi \leqslant x_0+\varepsilon,
	$
	such that the right-hand side of \eqref{eq-lim} becomes
	\begin{align*}
		\lim_{\varepsilon\to 0^+} \varphi(\xi)
		=
		\varphi(x_0)
		=
		-
		\int_{x_0}^\infty \varphi'(x){\rm d}x,
		%
	\end{align*}
	where, in the first equality, we used the continuity of $\varphi$.
\end{proof}

\begin{proposition}\label{prop-cop}
	Let
	$
	W_2(u,v)=\max\left\{0,u+v-1\right\}
	$
	and
	$
	M_2(u,v)=\min\{u,v\}
	$, $(u,v)\in(0,1)^2$,
	denote the lower and upper Fr\'echet-Hoeffding bounds, respectively.
	Then
	\begin{align*}
		{\partial^2 M_2(u,v)\over \partial u\partial v}
		=
		\delta_u(v),
		\quad 
		{\partial^2 W_2(u,v)\over \partial u\partial v}
		=
		\delta_{1-u}(v).	
	\end{align*}
\end{proposition}
\begin{proof}
	Note that $M_2$ and $W_2$ admit the representations
	\begin{align*}
		M_2(u,v)
		=
		\int_0^\infty
		\mathds{1}_{\{u\geqslant x\}}
		\mathds{1}_{\{v\geqslant x\}}
		{\rm d}x 
		=
		\int_{-\infty}^{\infty}
		H_0(x)H_x(u)H_x(v)
		{\rm d}x,
	\end{align*}
	and
	\begin{align*}
		W_2(u,v)
		=
		\int_0^\infty
		\mathds{1}_{\{u\geqslant x-v+1\}}
		{\rm d}x 
		=
		\int_{-\infty}^{\infty}
		H_0(x)H_{x-v+1}(u)
		{\rm d}x,
	\end{align*}
	respectively, where $H_{x_0}$ denotes the Heaviside step function (Example~\ref{ex-1}).
	
	Since differentiation is understood in the distributional sense and the distributional derivative is a continuous linear operator, differentiation with respect to $u$ and $v$ can be interchanged with integration with respect to the parameter $x$. Therefore, using the identity $H_{x_0}'=\delta_{x_0}$ (Example~\ref{ex-1}), we obtain
	\begin{align*}
		\frac{\partial^2 M_2(u,v)}{\partial u\partial v}
		=
		\int_{-\infty}^{\infty}
		H_0(x)
		\delta_x(u)
		\delta_x(v)
		{\rm d}x.
	\end{align*}
	The first identity then follows immediately from the sifting property of the Dirac delta.
	
	For the second identity, differentiating $W_2$ with respect to $u$ yields
	\begin{align*}
		\frac{\partial W_2(u,v)}{\partial u}
		&=
		\int_{-\infty}^{\infty}
		H_0(x)
		\delta_{x-v+1}(u)
		{\rm d}x.
	\end{align*}
	Applying the sifting property of the Dirac delta, we obtain
	\begin{align*}
		\frac{\partial W_2(u,v)}{\partial u}
		=
		H_0(u+v-1)
		=
		H_{1-u}(v).
	\end{align*}
	Differentiating the above identity with respect to $v$ and using once again that $H_{x_0}'=\delta_{x_0}$, we arrive at
	\begin{align*}
		\frac{\partial^2 W_2(u,v)}{\partial u \partial v}
		=
		\delta_{1-u}(v).
	\end{align*}
	This completes the proof.
\end{proof}

\section{Unified formulas for statistical quantities}
\label{Unified formulas for statistical quantities}

This section derives unified formulas for conditional expectations, conditional laws, hazard functions, and improper distributions using distributional representations based on the Dirac delta distribution.

\begin{theorem}\label{prop-2}
	Let $g:\mathbb{R}^2\to\mathbb{R}$ be a Borel measurable function and $X$ and $Y$ random variables with joint cumulative distribution function $F_{X,Y}$ such that the composite random variable $g(X,Y)$ is integrable. Then, for a fixed number $t$, it holds that
	\begin{align*}
		\lim_{\Delta t\to 0^+}
		{
			\mathbb{E}[g(X,Y)\mathds{1}_{\{t\leqslant Y<t+\Delta t\}}]
			\over \Delta t}
		=
		\int_{-\infty}^{\infty}
		\int_{-\infty}^{\infty}
		\delta_t(y)
		g(x,y)
		{\rm d}F_{X,Y}(x,y),
	\end{align*}
	where $\delta_t$ is the Dirac delta distribution.
\end{theorem}
\begin{proof}
	Since expectation is defined as a Lebesgue-Stieltjes integral, we can write
	\begin{align*}
		\mathbb{E}[g(X,Y)\mathds{1}_{\{t\leqslant Y<t+\Delta t\}}]
		&=
		\int_{-\infty}^{\infty}
		\int_{-\infty}^{\infty}
		\mathds{1}_{\{t\leqslant y<t+\Delta t\}}
		g(x,y)
		{\rm d}F_{X,Y}(x,y)
		\\[0,2cm]
		&=
		\int_{-\infty}^{\infty}
		\int_{-\infty}^{\infty}
		\left[
		H_{t}(y)
		-
		H_{t}(y-\Delta t)
		\right]
		g(x,y)
		{\rm d}F_{X,Y}(x,y),
	\end{align*}
	where $H_t$ is the Heaviside {step} function defined in Example \ref{ex-1}. Dividing the expression by $\Delta t$ and then taking the limit as $\Delta t\to 0^+$, we obtain
	\begin{align*}
		&\lim_{\Delta t\to 0^+}
		{
			\mathbb{E}[g(X,Y)\mathds{1}_{\{t\leqslant Y<t+\Delta t\}}]
			\over \Delta t}
		\\[0,2cm]
		&=
		\int_{-\infty}^{\infty}
		\int_{-\infty}^{\infty}
		\left[
		\lim_{\Delta t\to 0^+}
		{
			H_{t}(y)
			-
			H_{t}(y-\Delta t)
			\over 
			\Delta t}
		\right]
		g(x,y)
		{\rm d}F_{X,Y}(x,y).
	\end{align*}
	Hence, applying Proposition \ref{pro-main} the proof follows.
\end{proof}


\begin{proposition}\label{prop-cont-1}
	If $X$ and $Y$ are absolutely continuous random variables with joint probability density function $f_{X,Y}$, then
	\begin{align*}
		\lim_{\Delta t\to 0^+}
		{
			\mathbb{E}[g(X,Y)\mathds{1}_{\{t\leqslant Y<t+\Delta t\}}]
			\over \Delta t}
		=
		\int_{-\infty}^{\infty}
		g(x,t) f_{X,Y}(x,t)
		{\rm d}x.
	\end{align*}
\end{proposition}
\begin{proof}
	From Proposition \ref{prop-2}, we have
	\begin{align*}
		\lim_{\Delta t\to 0^+}
		{
			\mathbb{E}[g(X,Y)\mathds{1}_{\{t\leqslant Y<t+\Delta t\}}]
			\over \Delta t}
		=
		\int_{-\infty}^{\infty}
		\int_{-\infty}^{\infty}
		\delta_t(y)
		g(x,y) f_{X,Y}(x,y)
		{\rm d}x{\rm d}y.
	\end{align*}
	The proof follows from the sifting property of the Dirac delta.
\end{proof}

\begin{proposition}\label{prop-disc-1}
	If $X$ and $Y$ are discrete random variable assuming values in the set $\{(x_i,y_j):i\in I, j\in J\}$, where $I$ and $J$ are a countable index sets, then
	\begin{align*}
		\lim_{\Delta t\to 0^+}
		{
			\mathbb{E}[g(X,Y)\mathds{1}_{\{t\leqslant Y<t+\Delta t\}}]
			\over \Delta t}
		&=
		\sum_{i\in I} \sum_{j\in J} 
		g(x_i,y_j)
		\mathbb{P}(X=x_i, Y=y_j) \delta_{t}(y_j)
		\\[0,2cm]
		&=
		\begin{cases} 
			\displaystyle
			\sum_{i\in I} 
			g(x_i,t)
			\mathbb{P}(X=x_i, Y=t), & \text{if} \ t\in\text{Supp}(Y),
			\\[0,5cm]
			0, & \text{if} \ t\notin\text{Supp}(Y), 
		\end{cases}
	\end{align*}
\end{proposition}
\noindent where $\text{Supp}(Y)$ denotes the support of $Y$.
\begin{proof}
	Let $F_{X,Y}$ denote the joint cumulative distribution function of $(X,Y)^\top$. Observe that $F_{X,Y}$ may be expressed as
	\begin{align*}
		F_{X,Y}(x,y)
		=
		\sum_{i\in I} 
		\sum_{j\in J} 
		\mathbb{P}(X=x_i, Y=y_j) H_{x_i}(x) H_{y_j}(y),
		\quad x,y\in\mathbb{R},
	\end{align*}
	where $H_{x_i}$ and $H_{y_j}$ are the Heaviside {step} functions defined in Example \ref{ex-1}.
	
	By Theorem \ref{prop-2} and the linearity of the Lebesgue-Stieltjes integral with respect to the integrator, we obtain
	\begin{multline*}
		\lim_{\Delta t\to 0^+}
		{
			\mathbb{E}[g(X,Y)\mathds{1}_{\{t\leqslant Y<t+\Delta t\}}]
			\over \Delta t}
		\\[0,2cm]
		=
		\sum_{i\in I} 
		\sum_{j\in J} 
		\mathbb{P}(X=x_i, Y=y_j)
		\int_{-\infty}^{\infty}
		\int_{-\infty}^{\infty}
		\delta_t(y)
		g(x,y)
		{\rm d}H_{x_i}(x) {\rm d}H_{y_j}(y).
	\end{multline*}
	Since $H_{x_i}'=\delta_{x_i}$ and $H_{y_j}'=\delta_{y_j}$ (Example \ref{ex-1}), the previous expression can be rewritten as
	\begin{align*}
		\sum_{i\in I} 
		\sum_{j\in J} 
		\mathbb{P}(X=x_i, Y=y_j)
		\int_{-\infty}^{\infty}
		\int_{-\infty}^{\infty}
		\delta_t(y) \delta_{y_j}(y) 
		\delta_{x_i}(x)
		g(x,y)
		{\rm d}x {\rm d}y.
	\end{align*}
	The conclusion follows by the sifting property of the Dirac delta.
\end{proof}

Considering $g(x,y)=1$ in Theorem \ref{prop-2} as the unitary function, we obtain {the following proposition}.  
\begin{proposition}\label{prop-3}
	It holds that
	\begin{align*}
		\lim_{\Delta t\to 0^+}
		{
			\mathbb{P}(t\leqslant Y<t+\Delta t)
			\over \Delta t}
		=
		\int_{-\infty}^{\infty}
		\delta_t(y)
		{\rm d}F_{Y}(t).
	\end{align*}
\end{proposition}

With Proposition \ref{prop-3} in hand, we are able to establish the following results.
\begin{proposition}\label{prop-cont-2}
	If $Y$ is an absolutely continuous random variable with probability density function $f_Y$, then
	\begin{align*}
		\lim_{\Delta t\to 0^+}
		{
			\mathbb{P}(t\leqslant Y<t+\Delta t)
			\over \Delta t}
		=
		f_{Y}(t),
		\quad t\in{\rm Supp}(Y).
	\end{align*}
\end{proposition}
\begin{proof}
	By Proposition \ref{prop-3}, we obtain
	\begin{align*}
		\lim_{\Delta t\to 0^+}
		{
			\mathbb{P}(t\leqslant Y<t+\Delta t)
			\over \Delta t}
		=
		\int_{-\infty}^{\infty}
		\delta_t(y)
		f_{Y}(y)
		{\rm d}y.
	\end{align*}
	The result follows from the sifting property of the Dirac delta.
\end{proof}

\begin{proposition}\label{prop-disc-2}
	If $Y$ is a discrete random variable assuming values in the set $\{y_j:j\in J\}$, where $J$ is a countable index set, then
	\begin{align*}
		\lim_{\Delta t\to 0^+}
		{
			\mathbb{P}(t\leqslant Y<t+\Delta t)
			\over \Delta t}
		&=
		\sum_{j\in J}
		\mathbb{P}(Y=y_j)
		\delta_{t}(y_j)
		\\[0,2cm]
		&=
		\begin{cases} 
			\mathbb{P}(Y=t), & \text{if} \ t\in\text{Supp}(Y),
			\\[0,5cm]
			0, & \text{if} \ t\notin\text{Supp}(Y).
		\end{cases}
	\end{align*}
\end{proposition}
\begin{proof}
	Let $F_Y$ denote the cumulative distribution function of $Y$. Note that $F_Y$ can be written as
	\begin{align*}
		F_Y(y)
		=
		\sum_{j\in J} \mathbb{P}(Y=y_j) H_{y_j}(y),
		\quad y\in\mathbb{R},
	\end{align*}
	where $H_y$ denotes the Heaviside step function introduced in Example \ref{ex-1}.
	
	Using Proposition \ref{prop-3} together with the fact that the Lebesgue-Stieltjes integral is linear with respect to the integrator, we obtain
	\begin{align*}
		\lim_{\Delta t\to 0^+}
		{
			\mathbb{P}(t\leqslant Y<t+\Delta t)
			\over \Delta t}
		=
		\sum_{j\in J} \mathbb{P}(Y=y_j)
		\int_{-\infty}^{\infty}
		\delta_t(y)
		{\rm d}H_{y_j}(y).
	\end{align*}
	Using the fact that $H_{y_j}'=\delta_{y_j}$ (Example \ref{ex-1}), the above expression becomes 
	\begin{align*}
		\sum_{j\in J} \mathbb{P}(Y=y_j)
		\int_{-\infty}^{\infty}
		\delta_t(y)
		\delta_{y_j}(y)
		{\rm d}y.
	\end{align*}
	Applying the sifting property of the Dirac delta, the result follows.
\end{proof}

\subsection{Substitution principle for conditional expectations}\label{Substitution_principle}

From Propositions \ref{prop-cont-1}, \ref{prop-cont-2}, and \ref{prop-disc-1}, for a fixed $t\in\operatorname{Supp}(Y)$, the substitution principle for the conditional expectation of $g(X,Y)$ given $Y=t$ is given by
\begin{align*}
	\mathbb{E}[g&(X,Y)\vert Y=t]
	\equiv
	\lim_{\Delta t\to 0^+}
	\mathbb{E}[g(X,Y)\mid t\leqslant Y<t+\Delta t]	
	\\[0,2cm]	
	&=
	\lim_{\Delta t\to 0^+}
	{			\mathbb{E}[g(X,Y)\mathds{1}_{\{t\leqslant Y<t+\Delta t\}}]\over \mathbb{P}(t\leqslant Y<t+\Delta t)}
	\\[0,2cm]	
	&=
	\begin{cases}
		\displaystyle
		\int_{-\infty}^{\infty}
		g(x,t) f_{X\vert Y=t}(x)
		{\rm d}x, & \text{if} \ X \text{ and } Y \ \text{are absolutely continuous},
		\\[0,7cm]
		\displaystyle
		\sum_{i\in I} 
		g(x_i,t)
		p_{X\vert Y=t}(x_i),
		& \text{if} \ X \text{ and } Y \ \text{are discrete},
	\end{cases}
	\\[0,2cm]	
	&=
	\mathbb{E}[g(X,t)\vert Y=t],
	\nonumber
\end{align*}
where
\begin{align*}
	f_{X\vert Y=t}(x)\equiv
	{f_{X,Y}(x,t)\over f_Y(t)}
	\quad \text{and} \quad
	p_{X\vert Y=t}(x_i)\equiv {\mathbb{P}(X=x_i, Y=t)\over 	
		\mathbb{P}(Y=t)}.
\end{align*}

\begin{remark}
	The previous identity shows that conditioning on the event $\{Y=t\}$ may be interpreted as a substitution operation on the second argument of $g$. In this sense, the randomness associated with $Y$ is replaced by the deterministic value $t$, while the remaining uncertainty is entirely described by the conditional distribution of $X$ given $Y=t$.
\end{remark}

\subsection{Substitution principle for the  conditional law}\label{Substitution_principle_law}

Taking
$
g(X,Y)=\mathds{1}_{\{(X,Y)^\top\in \boldsymbol{B}\}},
$
in Subsection \ref{Substitution_principle}, where $\boldsymbol{B}$ is a Borel set in $\mathbb{R}^2$, we obtain, for a fixed $t\in\operatorname{Supp}(Y)$, the following substitution principle for the conditional law of $(X,Y)^\top$:
\begin{align*}
	\mathbb{P}(&(X,Y)^\top\in \boldsymbol{B}\vert Y=t)
	\equiv
	\lim_{\Delta t\to 0^+}
	\mathbb{P}((X,Y)^\top\in \boldsymbol{B}\mid t\leqslant Y<t+\Delta t)
	\\[0,2cm]	
	&=
	\lim_{\Delta t\to 0^+}
	{			\mathbb{P}((X,Y)^\top\in \boldsymbol{B} \mid \mathds{1}_{\{t\leqslant Y<t+\Delta t\}})\over \mathbb{P}(t\leqslant Y<t+\Delta t)}
	\\[0,2cm]	
	&=
	\begin{cases}
		\displaystyle
		\int_{B_t}
		f_{X\vert Y=t}(x)
		{\rm d}x, & \text{if} \ X \text{ and } Y \ \text{are absolutely continuous},
		\\[0,7cm]
		\displaystyle
		\sum_{i:x_i\in B_t} 
		p_{X\vert Y=t}(x_i),
		& \text{if} \ X \text{ and } Y \ \text{are discrete},
	\end{cases}
	\\[0,2cm]	
	&=
	\mathbb{P}((X,t)^\top\in \boldsymbol{B} \vert Y=t)
	\\[0,2cm]	
	&=
	\mathbb{P}(X\in B_t\vert Y=t),
	\nonumber
\end{align*}
where $f_{X\vert Y=t}$ and $p_{X\vert Y=t}$ are as given in Subsection \ref{Substitution_principle}, and $B_t\equiv\{x\in\mathbb{R}:(x,t)^\top\in \boldsymbol{B}\}$ is a Borel set in $\mathbb{R}$.

\begin{remark}
	The conditional distribution of $(X,Y)^\top$ given $Y=t$ is concentrated on the set
	$
	\{(x,t)^\top:x\in\mathbb{R}\}.
	$
	Consequently, conditioning on $Y=t$ reduces the two-dimensional distribution of $(X,Y)^\top$ to the one-dimensional conditional distribution of $X$ given $Y=t$.
\end{remark}

\subsection{The conditional  expectation given $Y=t$}

Combining Propositions \ref{prop-cont-1} (with $g\circ\pi_1$, where $\pi_1:\mathbb{R}^2\to\mathbb{R}$ denotes the projection onto the first coordinate, instead of $g$) and \ref{prop-cont-2}, as well as Propositions \ref{prop-disc-1} (with $g\circ\pi_1$ instead of $g$) and \ref{prop-disc-2}, we obtain, for a fixed $t\in\operatorname{Supp}(Y)$, that the conditional expectation of $g(X)$ given $Y=t$ is given by
\begin{align*}
	\mathbb{E}[g(X)&\vert Y=t]
	\equiv
	\lim_{\Delta t\to 0^+}
	\mathbb{E}[g(X)\mid t\leqslant Y<t+\Delta t]	
	\\[0,2cm]	
	&=
	\lim_{\Delta t\to 0^+}
	{			\mathbb{E}[g(X)\mathds{1}_{\{t\leqslant Y<t+\Delta t\}}]\over \mathbb{P}(t\leqslant Y<t+\Delta t)}
	\\[0,2cm]	
	&=
	\begin{cases}
		\displaystyle
		\int_{-\infty}^{\infty}
		g(x) f_{X\vert Y=t}(x)
		{\rm d}x, & \text{if} \ X \text{ and } Y \ \text{are absolutely continuous},
		\\[0,7cm]
		\displaystyle
		\sum_{i\in I} 
		g(x_i)
		p_{X\vert Y=t}(x_i),
		& \text{if} \ X \text{ and } Y \ \text{are discrete},
	\end{cases}
	\nonumber
\end{align*}
where $f_{X\vert Y=t}$ and $p_{X\vert Y=t}$ are as given in Subsection \ref{Substitution_principle}.

\begin{remark}
The previous formula may be viewed as a particular case of the substitution principle established in Subsection~\ref{Substitution_principle}, obtained by considering functions depending only on the first coordinate. It therefore provides a unified representation of conditional expectations in both discrete and absolutely continuous settings.
\end{remark}

\subsection{The hazard function}\label{The risk function}

From Propositions \ref{prop-cont-2} and \ref{prop-disc-2}, we obtain that, for a fixed value $t\in\operatorname{Supp}(Y)$, the hazard function for a nonnegative random variable $Y$ is given by
\begin{align*}
h_Y(t)
&\equiv
\lim_{\Delta t\to 0^+}
{\mathbb{P}(t\leqslant Y<t+\Delta t\mid Y\geqslant t)\over \Delta t}
\nonumber
\\[0,2cm]
&=
{1\over \mathbb{P}(Y\geqslant t)}  
\lim_{\Delta t\to 0^+}
{\mathbb{P}(t\leqslant Y<t+\Delta t)\over \Delta t}
\\[0,2cm]
&=
\begin{cases}
	\displaystyle
	{f_{Y}(t)\over S_{Y}(t)}, & \text{if} \ Y \ \text{is absolutely continuous},
	\\[0,7cm]
	{	\displaystyle
		\sum_{j\in J}
		\mathbb{P}(Y=y_j)
		\delta_{t}(y_j)\over 	
		\displaystyle
		\sum_{j\in J}
		\mathbb{P}(Y=y_j)
		\delta_{t}(y_j)+S_{Y}(t)},
	& \text{if} \ Y \ \text{is discrete},
\end{cases}
\nonumber
\end{align*}
where $S_Y$ denotes the survival function of $Y$.

\begin{remark}
The above representation shows that the hazard function may be interpreted as a localized rate of occurrence at time $t$ conditional on survival up to time $t$. The distributional formulation provides a common expression encompassing both continuous and discrete lifetime models.
\end{remark}

\subsection{Improper distribution}

In survival analysis, a cumulative distribution function $F$ is generally said to be improper (or defective) if
$
\lim_{t\to\infty} F(t) < 1.
$
Equivalently, the corresponding survival function
$
\overline F(t)=1-F(t)
$
satisfies
$
\lim_{t\to\infty} S(t) > 0.
$

Therefore, if $T$ denotes a lifetime random variable, then
$
\mathbb{P}(T=\infty)
=
1-\lim_{t\to\infty}F(t)
>0.
$

Hence, there is a positive probability that the event of interest never occurs.

An example of an improper survival function can be seen in the Berkson and Gage mixture models \citep{Berkson1952}, commonly represented as
\[
\overline F(t)
=
p+(1-p)\overline F_0(t),
\quad 0<p<1,
\]
where $p=\mathbb{P}(T=\infty)$ is called the cure fraction, and $\overline{F}_0$ is a (proper) survival function satisfying
$
\lim_{t\to\infty} \overline F_0(t)=0.
$
Consequently,
$
\lim_{t\to\infty} \overline F(t)=p.
$

Hence, $F$ is an improper probability distribution on $[0,\infty)$. Nevertheless, it becomes a (proper) distribution when the state space is extended to include the point $\infty$, that is, to the extended nonnegative real line $\overline{\mathbb{R}^+_0}=[0,\infty)\cup\{\infty\}$.

Indeed, consider the stochastic representation of a mixture composed of a degenerate random variable at $\infty$ and a random variable $Z$, namely,
\begin{align*}
T\stackrel{d}{=}
\begin{cases}
	\infty, & \text{if} \ B=1,
	\\
	Z, & \text{if} \ B=0,
\end{cases}
\end{align*}
where $Z$ is independent of $B\sim\text{Bernoulli}(p)$, and $\mathbb{P}(Z\leqslant t)=1-\overline F_0(t)$. In the above, $\stackrel{d}{=}$ denotes equality in distribution of random variables.

Using the law of total probability, for any $t\in\overline{\mathbb{R}^+_0}$,
\begin{align*}
F_T(t)
\equiv 
\mathbb{P}(T\leqslant t)
&=
\mathbb{P}(T\leqslant t\mid B=1)\mathbb{P}(B=1)
+
\mathbb{P}(T\leqslant t\mid B=0)\mathbb{P}(B=0)
\\
&=
p
H_{\infty}(t)
+
(1-p)
\mathbb{P}(Z\leqslant t)
\\
&=
p
H_{\infty}(t)
+
(1-p)[1-\overline F_0(t)]
\\[0,2cm]
&=	
\begin{cases}
	0, & t<0,\\
	(1-p)[1-\overline F_0(t)], & 0\leqslant t<\infty,\\
	1, & t=\infty,
\end{cases}
\end{align*}
where $H_y$ is the Heaviside step function as given in Example \ref{ex-1}.

Therefore, \(F_T\) defines a (proper) distribution function on the extended real line $\overline{\mathbb{R}^+_0}$.

	%
	%
	%
	%
	%
	%
	%
%

Moreover, the survival function of $T$ is given by
\begin{align*}
\overline F_T(t)
&=
p
[1-H_{\infty}(t)]
+
(1-p)\overline F_0(t)
\\[0,2cm]
&=
\begin{cases}
	1, & t<0,\\
	p
	+
	(1-p)\overline F_0(t), & 0 \leqslant t<\infty,\\
	0, & t=\infty.
\end{cases}
\end{align*}

\begin{remark}
The preceding construction shows that every improper distribution on $[0,\infty)$ may be regarded as a proper probability distribution on the extended state space
$
\overline{\mathbb{R}^+_0}=[0,\infty)\cup\{\infty\}.
$
From this perspective, defective distributions arise naturally from probability measures that assign positive mass to the point at infinity.
\end{remark}

In summary, the distinction between proper and improper distributions is merely terminological. A function either satisfies the axioms of a probability distribution or it does not. What is commonly referred to as an ``improper distribution'' is simply a function obtained by restricting the support of an underlying distribution, thereby reducing its total mass below one, as naturally expected.

\begin{remark}
The results obtained in this section reveal a common structure underlying conditional expectations, conditional laws, hazard functions, and improper distributions. In each case, the conditioning operation is represented through a localization mechanism induced by the Dirac delta distribution, providing a unified treatment of discrete and absolutely continuous random variables.
\end{remark}

\section{Wasserstein distance formulas}\label{Wasserstein distance formulas}

This section establishes moment bounds and quantile representations for Wasserstein-type functionals through copula-based extremal arguments.

Let $(X,Y)^\top$ be an $\mathbb{R}^2$-valued random vector with joint distribution function $F_{X,Y}$ and marginal distribution functions $F_X$ and $F_Y$. By Sklar's theorem \citep{Sklar-1959}, there exists a copula $C$ such that
\[
F_{X,Y}(x,y)=C\big(F_X(x),F_Y(y)\big),
\quad (x,y)^\top\in\mathbb{R}^2.
\]

Let $f:\mathbb{R}^2\to\mathbb{R}$ be a Borel measurable function. We are interested in the expectation
$
\mathbb{E}_C[f(X,Y)],
$
particularly in its monotonicity properties over suitable classes of copulas, since these properties yield lower and upper bounds for $\mathbb{E}_C[f(X,Y)]$.

Assuming that the marginal distributions are fixed, the expectation depends only on the copula $C$. Accordingly, the expectation operator can be written as
\begin{align}
\pi_f(C)
\equiv
\mathbb{E}[f(X,Y)]
&=
\mathbb{E}_{C}[f(X,Y)]
\nonumber
\\[0,2cm]
&=
\int_{\mathbb{R}^2} f(x,y){\rm d}C(F_{X}(x),F_{Y}(y))
\nonumber
\\[0,2cm]
&=
\int_{(0,1)^2}
f(F_{X}^{-1}(u),F_{Y}^{-1}(v))
{\rm d} C(u,v). \label{id-exp-cop}
\end{align}
\begin{definition}[\cite{Lux2017}]
\label{antitonic-def}
A function $f: \mathbb{R}^2\to \mathbb{R}$ is called $\Delta$-antitonic if for every square $[a_1,b_1]\times[a_2,b_2]\subset\mathbb{R}^2$, with $a_j<b_j$ for $j=1,2$, it holds that
\begin{align*}
	(\Delta^1_{a_1,b_1}\circ\Delta^2_{a_2,b_2})f(x,y)\geqslant 0 \quad \text{for all} \ (x,y)^\top\in \mathbb{R}^2,
\end{align*}
where we are adopting the notations
\begin{align*}
	\Delta^1_{a,b}(f(x,y))=f(b,y)-f(a,y),
	\quad
	\Delta^2_{a,b}(f(x,y))=f(x,b)-f(x,a).
\end{align*}
The operator $\Delta$ is commonly known as finite difference operator.
\end{definition}


%
\begin{remark}\label{rem-antitonic}
If $f: \mathbb{R}^2\to \mathbb{R}$ is at least twice differentiable, note that
\begin{align*}
	(\Delta^1_{a_1,b_1}\circ\Delta^2_{a_2,b_2})f(x,y)
	&
	=
	\Delta^1_{a_1,b_1}(f(x,b_2)-f(x,a_2))
	\\[0,2cm]
	&=
	[f(b_1,b_2)-f(b_1,a_2)]
	-
	[f(a_1,b_2)-f(a_1,a_2)]
	\\[0,2cm]
	&=	
	\int^{b_2}_{a_2}\int^{b_1}_{a_1}
	{\partial^2 f(x,y)\over\partial x\partial y}\,
	{\rm d}x
	{\rm d}y.
\end{align*}
Hence, a sufficient condition for $f$ to be $\Delta$-antitonic is that
\begin{align*}
	{\partial^2 f(x,y)\over\partial x\partial y}\geqslant 0, \quad \forall (x,y)^\top\in(a_1,b_1)\times(a_2,b_2). 
\end{align*}
\end{remark}
%
%

\begin{proposition}\label{antitonic-function}
%
The function $-f$, where $f(x,y)=\vert x-y\vert^r$, with $r\geqslant 1$, is $\Delta$-antitonic.
\end{proposition}
\begin{proof}
%
%
	Let $\psi(t)=|t|^r$, $r\geqslant 1$. Since $\psi$ is convex, the increment
	$\psi(t+h)-\psi(t)$ is nondecreasing in $t$ for every $h>0$. Thus, setting $h=b_1-a_1>0$, we have
	\[
	\begin{aligned}
		(\Delta^1_{a_1,b_1}\circ\Delta^2_{a_2,b_2})(-f)
		&=
		(|b_1-a_2|^r-|a_1-a_2|^r)
		-
		(|b_1-b_2|^r-|a_1-b_2|^r)
		\\[0,2cm]
		&=
		[\psi(a_1-a_2+h)-\psi(a_1-a_2)]
		-[\psi(a_1-b_2+h)-\psi(a_1-b_2)]
		\geqslant 0,
	\end{aligned}
	\]
	because $a_1-a_2\geqslant a_1-b_2$. Therefore, $-f$ is $\Delta$-antitonic.
\end{proof}

\begin{remark}\label{rem-imp}
Other $\Delta$-antitonic functions include
$f(x,y)=xF_Y(y)$,
$f(x,y)=C(F_X(x),F_Y(y))$,
and
$f(x,y)=F_X(x)F_Y(y)$.
In contrast, the functions
$f(x,y)=\mathds{1}_{\{x>y\}}$
and
$f(x,y)=\mathds{1}_{\{x\neq y\}}$
are not $\Delta$-antitonic.
\end{remark}

To state the following result, let (for $0<u,v<1$)
\[
W_2(u,v)=\max\left\{0,u+v-1\right\}
\quad
\text{and}
\quad
M_2(u,v)=\min\{u,v\},
\]
denote the lower and upper Fr\'echet-Hoeffding bounds, respectively.
It is well known that $W_2(u,v)$ and $M_2(u,v)$ are copulas
\citep{Genest1999}.
\begin{proposition}\label{antitonic}
If $f:\mathbb{R}^2\to
\mathbb{R}$ is $\Delta$-antitonic, then $\pi_f(C)$ is monotonically increasing in $C$, and besides
\begin{align*}
	\pi_f(W_2)
	\leqslant
	\pi_f(C)
	\leqslant
	\pi_f(M_2).
\end{align*}

In this setting, if $-f$ is $\Delta$-antitonic, then all inequalities in the above equation are reversed.
%
\end{proposition}
\begin{proof}
See Theorem 5.5 and Proposition 6.1 of \cite{Lux2017}.
\end{proof}

%
\begin{theorem}\label{main-result}
If the function $f:\mathbb{R}^2\to
\mathbb{R}$ is $\Delta$-antitonic, then
\begin{align*}
	\int_{0}^{1} 
	f(F_X^{-1}(u),F_Y^{-1}(1-u))
	{\rm d} u
	\leqslant
	\mathbb{E}[f(X,Y)]
	\leqslant
	\int_{0}^{1} 
	f(F_X^{-1}(u),F_Y^{-1}(u))
	{\rm d} u.
\end{align*}

In this setting, if $-f$ is $\Delta$-antitonic, then all inequalities in the above equation are reversed.
\end{theorem}
\begin{proof}
%
If $f$ is $\Delta$-antitonic, by Proposition \ref{antitonic}, we have
$\pi_f(M_2)
\geqslant
\pi_f(C)
\geqslant
\pi_f(W_2)$. 
Then, by using  \eqref{id-exp-cop}  consecutively, we get
\begin{align}\label{last-integral}
	\mathbb{E}[f(X,Y)]
	\stackrel{\eqref{id-exp-cop} }{=}
	\pi_f(C)
	\nonumber
	&\leqslant 	
	\pi_f(M_2)
	\nonumber
	\\[0,2cm]
	&\stackrel{\eqref{id-exp-cop} }{=}
	\mathbb{E}_{M_2}[f(X^*,Y^*)]
	\nonumber
	\\[0,2cm]
	&\stackrel{\eqref{id-exp-cop} }{=}
	\int_{(0,1)^2}
	f(F_{X}^{-1}(u),F_{Y}^{-1}(v))
	{\rm d} M_2(u,v)
	\nonumber
	\\[0,2cm]
	&=
	\int_{(0,1)^2}
	f(F_{X}^{-1}(u),F_{Y}^{-1}(v))\,
	{\partial^2 M_2(u,v)\over \partial u\partial v}\,
	{\rm d} v 
	{\rm d} u,
\end{align}
where $(X^*,Y^*)^\top$ is a random vector with joint distribution $F_{X^*,Y^*}(x,y) = M_2(F_{X}(x),F_{Y}(y))$ for all $(x,y)^\top\in\mathbb{R}^2$ and marginals $F_X$ and $F_Y$.

In an analogous way, we have
\begin{align}\label{last-integral-1}
	\mathbb{E}[f(X,Y)]
	&\geqslant 	
	\mathbb{E}_{W_2}[f(X',Y')]
	\nonumber
	\\[0,2cm]
	&=
	\int_{(0,1)^2}
	f(F_{X}^{-1}(u),F_{Y}^{-1}(v))\,
	{\partial^2 W_2(u,v)\over \partial u\partial v}\,
	{\rm d} v
	{\rm d} u,
\end{align}
where $(X',Y')^\top$ is a random vector with joint distribution $F_{X',Y'}(x,y) = W_2(F_{X}(x),F_{Y}(y))$ for all $(x,y)^\top\in\mathbb{R}^2$ and marginals $F_X$ and $F_Y$.

Since (see Proposition \ref{prop-cop})
\begin{align*}
	{\partial^2 M_2(u,v)\over \partial u\partial v}
	=
	\delta_u(v),
	\quad 
	{\partial^2 W_2(u,v)\over \partial u\partial v}
	=
	\delta_{1-u}(v),	
\end{align*}
where $\delta_x$ is the Dirac delta function,
the integrals in \eqref{last-integral} and \eqref{last-integral-1} are written as
\begin{align*}
	\int_{0}^{1} 
	f(F_X^{-1}(u),F_Y^{-1}(u))
	{\rm d} u,
	\quad 
	\int_{0}^{1} 
	f(F_X^{-1}(u),F_Y^{-1}(1-u))
	{\rm d} u,
\end{align*}
respectively. 
Hence, by combining the above two integrals with inequalities \eqref{last-integral} and \eqref{last-integral-1}, the proof of first part of theorem follows readily.

The second part, namely the proof that $-f$ is $\Delta$-antitonic, follows analogously and is therefore omitted.
\end{proof}

\begin{corollary}
If $f$ is $\Delta$-antitonic, then
\begin{multline*}
	\int_{0}^{1} 
	f(F_X^{-1}(u),F_Y^{-1}(1-u))
	{\rm d} u
	\leqslant
	\int_{(0,1)^2}
	f(F_X^{-1}(u),F_Y^{-1}(v))
	{\rm d} v
	{\rm d} u
	\\[0,2cm]
	\leqslant
	\int_{0}^{1} 
	f(F_X^{-1}(u),F_Y^{-1}(u))
	{\rm d} u.
\end{multline*}

In this setting, if $-f$ is $\Delta$-antitonic, then all inequalities in the above equation are reversed.
\end{corollary}
\begin{proof}
Since $\mathbb{E}_{I_2}[f(X,Y)]
=
\int_{(0,1)^2}
f(F_X^{-1}(u),F_Y^{-1}(v))
{\rm d} v
{\rm d} u$,
where $I_2$ denotes the independence copula, that is, $I_2(u,v)=uv$,
the proof follows readily by applying Theorem \ref{main-result}.
\end{proof}

\begin{remark}
The scope of Theorem~\ref{main-result} is considerably broader. For instance, by choosing the $\Delta$-antitonic functions
$f(x,y)=xF_Y(y)$,
$f(x,y)=C(F_X(x),F_Y(y))$,
and
$f(x,y)=F_X(x)F_Y(y)$
(see Remark~\ref{rem-imp}),
one obtains bounds for the Gini correlation \citep{Schechtman-Yitzhaki1987}, Kendall's $\tau$ \citep{Idaa2014}, and Spearman's $\rho_s$ \citep{Idaa2014}, respectively.

In the remainder of the paper, we restrict attention to the $\Delta$-antitonic function
$-f(x,y)=-|x-y|^r$, $r\geqslant 1$,
since it gives rise to the absolute difference moments
$\mathbb E|X-Y|^r$, which constitute the main object of study in this work and encompass important quantities such as Wasserstein distances and Gini-type measures.
\end{remark}

\subsection{Moment bounds via coupling functionals}

For $r>0$, define the lower and upper transportation functionals of order $r$ by \cite{Rachev1998}
\[
\mathcal W_r(F,G)
=
\inf_{(X,Y)^\top}
\bigl\{\mathbb E|X-Y|^r\bigr\}^{1/r},
\quad
\overline{\mathcal W}_r(F,G)
=
\sup_{(X,Y)^\top}
\bigl\{\mathbb E|X-Y|^r\bigr\}^{1/r},
\]
where the infimum and supremum are taken over all random vectors $(X,Y)^\top$ with marginal distributions $F$ and $G$, respectively. The functional $\mathcal W_r(F,G)$ is precisely the Wasserstein distance of order $r$ between $F$ and $G$ \cite{Mallows1972,Villani2009}.
\begin{remark}
For $0<r<1$, the quantities
$
\mathcal{W}_r(F,G)
$ 
and
$\overline{\mathcal W}_r(F,G)$
are not a metric on the space of probability distributions. 
\end{remark}

By combining Theorem \ref{main-result} and Proposition \ref{antitonic-function}, we get 
\begin{proposition}\label{bounds-expectation}
If $(X,Y)^\top$ is an $\mathbb{R}^2$-valued random vector with joint distribution $F_{X,Y}$ and marginals $F_{X}$ and $F_Y$, then
\begin{align*}
	\int_{0}^{1} 
	\vert F_X^{-1}(u)-F_Y^{-1}(u)\vert^r
	{\rm d} u
	\leqslant
	\mathbb{E}\vert X-Y\vert^r
	\leqslant
	\int_{0}^{1} 
	\vert F_X^{-1}(u)-F_Y^{-1}(1-u)\vert^r
	{\rm d} u,
\end{align*}
{for all} $r\geqslant 1$.

\begin{proposition}\label{ine-geral-r}
	Under the assumptions of Proposition~\ref{bounds-expectation}, we have  ({for all} $r\geqslant 1$)
	\begin{align*}
		\int_{-\infty}^{\infty}
		|\overline F_X(t)-\overline F_Y(t)|^r
		\,{\rm d}t
		\leqslant
		\mathbb E|X-Y|^r
		\leqslant
		\int_{-\infty}^{\infty}
		\left[
		1-
		|\overline F_X(t)+\overline F_Y(t)-1|^r
		\right]
		{\rm d}t,
	\end{align*} 
	where $\overline F_Z=1-F_Z$ denotes the survival function of the random
	variable $Z$.
\end{proposition}
\begin{proof}
	Using the identity
	\begin{align}\label{id-difference}
		|x-y|
		=
		\int_{-\infty}^{\infty}
		\bigl|
		\mathds{1}_{\{x\leqslant t\}}
		-
		\mathds{1}_{\{y\leqslant t\}}
		\bigr|
		\,{\rm d}t,
	\end{align}
	we can write
	\begin{align*}
		\int_{0}^{1}
		|F_X^{-1}(u)-F_Y^{-1}(1-u)|^r
		\,{\rm d}u
		&=
		\int_{0}^{1}
		\int_{-\infty}^{\infty}
		\bigl|
		\mathds{1}_{\{F_X^{-1}(u)\leqslant t\}}
		-
		\mathds{1}_{\{F_Y^{-1}(1-u)\leqslant t\}}
		\bigr|^r
		\,{\rm d}t\,{\rm d}u
		\\[0.2cm]
		&=
		\int_{-\infty}^{\infty}
		\int_{0}^{1}
		\bigl|
		\mathds{1}_{\{u\leqslant F_X(t)\}}
		-
		\mathds{1}_{\{1-u\leqslant F_Y(t)\}}
		\bigr|^r
		\,{\rm d}u\,{\rm d}t
		\\[0.2cm]
		&=
		\int_{-\infty}^{\infty}
		\int_{0}^{1}
		\bigl|
		1-\mathds{1}_{\{F_X(t)\leqslant u\}}
		-\mathds{1}_{\{1-F_Y(t)\leqslant u\}}
		\bigr|^r
		\,{\rm d}u\,{\rm d}t,
	\end{align*}
	where, in the second equality, the order of integration is exchanged by
	Fubini's theorem. Using the identity
	$
	|1-p-q|^r=1-|p-q|^r,
	\ p,q\in\{0,1\},
	$
	the above expression becomes
	\begin{align*}
		\int_{-\infty}^{\infty}
		\int_{0}^{1}
		\left[
		1-
		\bigl|
		\mathds{1}_{\{F_X(t)\leqslant u\}}
		-
		\mathds{1}_{\{1-F_Y(t)\leqslant u\}}
		\bigr|^r
		\right]
		\,{\rm d}u\,{\rm d}t.
	\end{align*}
	Applying \eqref{id-difference} once again yields
	\begin{align*}
		\int_{-\infty}^{\infty}
		\left[
		1-
		|F_X(t)+F_Y(t)-1|^r
		\right]
		\,{\rm d}t.
	\end{align*}
	Therefore, we have established that
	\begin{align}\label{1-id}
		\int_{0}^{1}
		|F_X^{-1}(u)-F_Y^{-1}(1-u)|^r
		\,{\rm d}u
		=
		\int_{-\infty}^{\infty}
		\left[
		1-
		|F_X(t)+F_Y(t)-1|^r
		\right]
		\,{\rm d}t.
	\end{align}
	
	Similarly, using the identity \eqref{id-difference} and the Fubini's theorem, one can show that
	\begin{align}\label{2-id}
		\int_{0}^{1}
		|F_X^{-1}(u)-F_Y^{-1}(u)|^r
		\,{\rm d}u
		&=
		\int_{-\infty}^{\infty}
		\int_{0}^{1}
		\bigl|
		\mathds{1}_{\{F_X(t)\leqslant u\}}
		-\mathds{1}_{\{F_Y(t)\leqslant u\}}
		\bigr|^r
		\,{\rm d}u\,{\rm d}t
		\nonumber
		\\[0,2cm]
		&=
		\int_{-\infty}^{\infty}
		|F_X(t)-F_Y(t)|^r
		\,{\rm d}t.
	\end{align}
	
	Consequently, combining \eqref{1-id} and \eqref{2-id} with
	Proposition \ref{bounds-expectation} completes the proof.
\end{proof}	

\begin{remark}
	Proposition \ref{ine-geral-r} shows that the $r$-th absolute difference
	moment $\mathbb E|X-Y|^r$ admits bounds expressed solely in terms of the
	marginal survival functions $\overline F_X$ and $\overline F_Y$. The
	lower bound quantifies their pointwise discrepancy, whereas the upper
	bound captures their pointwise overlap. Thus, both bounds are completely
	determined by the marginal survival functions.
\end{remark}

\begin{definition}[\cite{Capaldo2025}]
	Let $(X,Y)^\top$ be a random vector with finite means. The {bivariate Gini's mean difference} of $(X,Y)^\top$ is defined by
	\[
	\operatorname{GMD}(X,Y)=\mathbb{E}|X-Y|.
	\]
	
	If, in addition, $(X,Y)^\top$ is nonnegative, the {bivariate Gini's index} of $(X,Y)^\top$ is defined by
	\[
	G(X,Y)
	=
	\frac{\operatorname{GMD}(X,Y)}
	{\mathbb{E}[X]+\mathbb{E}[Y]}.
	\]
\end{definition}

\begin{remark}
	Taking $r=1$ in Proposition \ref{ine-geral-r} yields
	\begin{align*}
		\int_{-\infty}^{\infty}
		|\overline F_X(t)-\overline F_Y(t)|
		\,{\rm d}t
		\leqslant
		\operatorname{GMD}(X,Y)
		\leqslant
		\int_{-\infty}^{\infty}
		\left[
		1-
		|\overline F_X(t)+\overline F_Y(t)-1|
		\right]
		\,{\rm d}t,
	\end{align*}
	and 
	\begin{align*}
		{\int_{-\infty}^{\infty}
			|\overline F_X(t)-\overline F_Y(t)|
			\,{\rm d}t\over \mathbb{E}[X]+\mathbb{E}[Y]}
		\leqslant
		G(X,Y)
		\leqslant
		{
			\int_{-\infty}^{\infty}
			\left[
			1-
			|\overline F_X(t)+\overline F_Y(t)-1|
			\right]
			\,{\rm d}t \over \mathbb{E}[X]+\mathbb{E}[Y]}.
	\end{align*}
	
	The above two inequalities were previously established in Proposition~12 of \cite{Capaldo2025}.
\end{remark}

%
\end{proposition}


The following theorem gives the classical quantile representation of the Wasserstein distance.

\begin{theorem}\label{Wasserstein distance}
Under the assumptions of Proposition~\ref{bounds-expectation}, for every $r\geqslant 1$,
\[
\mathcal W_r^r(F_X,F_Y)
=
\int_0^1
|F_X^{-1}(u)-F_Y^{-1}(u)|^r{\rm d}u
=
\int_{-\infty}^{\infty}
|F_X(t)-F_Y(t)|^r
\,{\rm d}t.
\]
\end{theorem}
\begin{proof}
The second identity stated in the theorem follows directly from \eqref{2-id}. For the first identity, Proposition \ref{bounds-expectation} implies that
\[
\int_0^1
|F_X^{-1}(u)-F_Y^{-1}(u)|^r{\rm d}u
\leqslant
\mathbb E|X-Y|^r,
\]
for every coupling $(X,Y)^\top$ of $F_X$ and $F_Y$. Taking the infimum over all such couplings yields
\[
\int_0^1
|F_X^{-1}(u)-F_Y^{-1}(u)|^r{\rm d}u
\leqslant
\mathcal W_r^r(F_X,F_Y).
\]

Conversely, if $U\sim U(0,1)$ and
$
X^*=F_X^{-1}(U),
\
Y^*=F_Y^{-1}(U),
$
then $(X^*,Y^*)^\top$ is a coupling of $F_X$ and $F_Y$. Therefore,
\[
\mathcal W_r^r(F_X,F_Y)
\leqslant
\mathbb E|X^*-Y^*|^r
=
\int_0^1
|F_X^{-1}(u)-F_Y^{-1}(u)|^r{\rm d}u.
\]
This completes the proof.
\end{proof}

\begin{remark}
Theorem~\ref{Wasserstein distance} was previously established by
\cite{Dorea2012}. Our proof relies on copula theory and the extremal
properties of the Fr\'echet--Hoeffding bounds, highlighting the role of
dependence structures in Wasserstein-type functionals.
\end{remark}

The following theorem provides a quantile representation for the upper
transportation functional.

\begin{theorem}\label{Wasserstein distance-1}
Under the assumptions of Proposition~\ref{bounds-expectation}, for every
$r\geqslant 1$,
\[
\overline{\mathcal W}_r^{\,r}(F_X,F_Y)
=
\int_0^1
|F_X^{-1}(u)-F_Y^{-1}(1-u)|^r{\rm d}u
=
\int_{-\infty}^{\infty}
\left[
1-
|F_X(t)+F_Y(t)-1|^r
\right]
\,{\rm d}t
.
\]
\end{theorem}

\begin{proof}
The second identity is an immediate consequence of \eqref{1-id}. For the first identity, Proposition~\ref{bounds-expectation} yields
\[
\mathbb E|X-Y|^r
\leqslant
\int_0^1
|F_X^{-1}(u)-F_Y^{-1}(1-u)|^r{\rm d}u
\]
for every coupling $(X,Y)^\top$ of $F_X$ and $F_Y$. Taking the supremum over all
such couplings yields
\[
\overline{\mathcal W}_r^{\,r}(F_X,F_Y)
\leqslant
\int_0^1
|F_X^{-1}(u)-F_Y^{-1}(1-u)|^r{\rm d}u.
\]

Conversely, if $U\sim U(0,1)$ and
$
X^*=F_X^{-1}(U),
\
Y^*=F_Y^{-1}(1-U),
$
then $(X^*,Y^*)^\top$ is a coupling of $F_X$ and $F_Y$. Hence,
\[
\overline{\mathcal W}_r^{\,r}(F_X,F_Y)
\geqslant
\mathbb E|X^*-Y^*|^r
=
\int_0^1
|F_X^{-1}(u)-F_Y^{-1}(1-u)|^r{\rm d}u.
\]
This completes the proof.
\end{proof}


Combining Proposition \ref{bounds-expectation} with Theorems \ref{Wasserstein distance} and \ref{Wasserstein distance-1}, we obtain
\begin{proposition}\label{bounds-expectation-1}
Under the assumptions of Proposition~\ref{bounds-expectation}, the following inequality holds for every $r\geqslant 1$:
\begin{align*}
	\mathcal{W}_r^r(F_X, F_Y) 
	\leqslant
	\mathbb{E}\vert X-Y\vert^r
	\leqslant
	\overline{\mathcal W}_r^{\,r}(F_X,F_Y).
\end{align*}
\end{proposition}

\begin{remark}
The previous proposition characterizes the range of $\mathbb E|X-Y|^r$ among all couplings of $F_X$ and $F_Y$. Hence, the lower and upper transportation functionals provide the best possible distribution-free bounds for the $r$th absolute moment of $X-Y$, depending only on the marginal distributions.
\end{remark}

\begin{example}
Let $X\sim\mathrm{Bernoulli}(p)$ and $Y\sim\mathrm{Bernoulli}(q)$, with
$0<p\leqslant q<1$. Since $|X-Y|\in\{0,1\}$, we have
$
|X-Y|^r=|X-Y|,
\ r\geqslant 1.
$
Moreover,
\[
F_X^{-1}(u)=
\begin{cases}
	0, & 0<u\leqslant 1-p,\\
	1, & 1-p<u<1,
\end{cases}
\quad
F_Y^{-1}(u)=
\begin{cases}
	0, & 0<u\leqslant 1-q,\\
	1, & 1-q<u<1.
\end{cases}
\]

Therefore,
\[
\mathcal W_r^r(F_X,F_Y)
=
\int_0^1
|F_X^{-1}(u)-F_Y^{-1}(u)|{\rm d}u
=
q-p,
\]
and
\[
\overline{\mathcal W}_r^{\,r}(F_X,F_Y)
=
\int_0^1
|F_X^{-1}(u)-F_Y^{-1}(1-u)|{\rm d}u
=
\begin{cases}
	p+q, & p+q\leqslant 1,\\
	2-(p+q), & p+q>1.
\end{cases}
\]

Hence, Proposition~\ref{bounds-expectation-1} gives
\[
q-p
\leqslant
\mathbb E|X-Y|^r
\leqslant
\begin{cases}
	p+q, & p+q\leqslant 1,\\
	2-(p+q), & p+q>1.
\end{cases}
\]

For example, if $p=2/10$ and $q=6/10$, then
$
4/10
\leqslant
\mathbb E|X-Y|^r
\leqslant 
8/10.
$

Thus, even without any information about the dependence structure between $X$ and $Y$, Proposition \ref{bounds-expectation-1} determines the exact range of possible values of $\mathbb E|X-Y|^r$ using only the marginal distributions.
\end{example}

The inequalities in Proposition \ref{bounds-expectation-1} motivates the introduction of a normalized dependence coefficient based on the $r$th absolute moment. Indeed,  whenever
$
\mathcal W_r(F_X,F_Y)
<
\overline{\mathcal W}_r(F_X,F_Y),
$
we may define
\begin{align}\label{rho}
\rho_r
\equiv
\rho_r(X,Y)
=
\frac{\mathbb E|X-Y|^r-\mathcal W_r^r(F_X,F_Y)}
{\overline{\mathcal W}_r^{\,r}(F_X,F_Y)-\mathcal W_r^r(F_X,F_Y)}\in[0,1].
\end{align}

Equivalently,
\[
\mathbb E|X-Y|^r
=
(1-\rho_r)
\mathcal W_r^r(F_X,F_Y)
+
\rho_r
\overline{\mathcal W}_r^{\,r}(F_X,F_Y),
\]
showing that $\rho_r$ measures the relative position of
$\mathbb E|X-Y|^r$
between its minimum and maximum attainable values under fixed marginals.
\begin{figure}[ht]
\centering
\begin{tikzpicture}[scale=1.1]
	
	\draw[->] (-1.5,0) -- (8,0) node[right] {$\rho_r$};
	\draw[->] (-1.5,0) -- (-1.5,4) node[above] {$\mathbb{E}|X-Y|^r$};
	
	\draw[thick] (1,1) -- (7,3);
	
	\fill (1,1) circle (2pt);
	\fill (7,3) circle (2pt);
	
	\node[left] at (1,1)
	{$\mathcal W_r^r(F_X,F_Y)$};
	
	\node[right] at (7,3)
	{$\overline{\mathcal W}_r^{\,r}(F_X,F_Y)$};
	
	\draw[dashed] (1,0) -- (1,1);
	\draw[dashed] (7,0) -- (7,3);
	
	\node[below] at (1,0) {$0$};
	\node[below] at (7,0) {$1$};
	
	\fill (4,2) circle (2pt);
	\draw[dashed] (4,0) -- (4,2);
	
	\node[below] at (4,0) {$\rho_r^*$};
	\node[above] at (4,2.2)
	{$\mathbb E|X-Y|^r$};
	
\end{tikzpicture}
\caption{The $r$th absolute moment as a function of the dependence parameter $\rho_r$. The quantity $\rho_r^*$ is a fixed value of $\rho_r$.}
\label{Fig-1}
\end{figure}
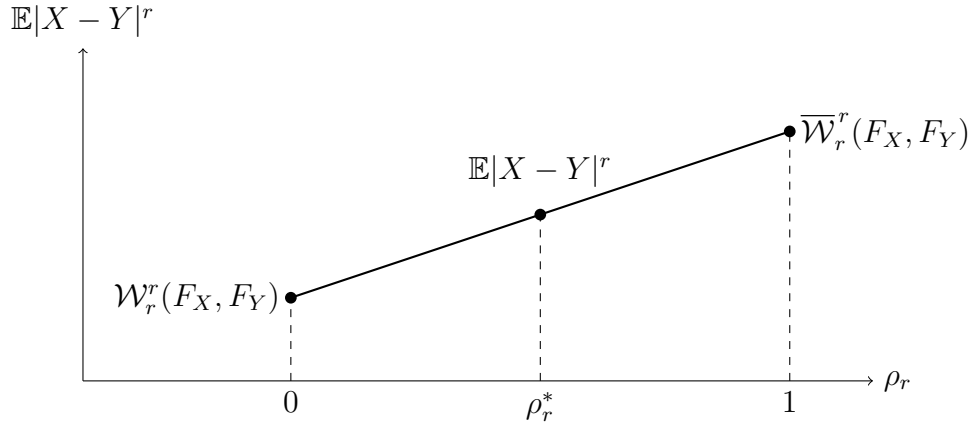

\begin{remark}
The coefficient $\rho_r$ measures the relative position of
$\mathbb E|X-Y|^r$
within the interval
$
[
\mathcal W_r^r(F_X,F_Y),
\overline{\mathcal W}_r^{\,r}(F_X,F_Y)
]
$ (Figure \ref{Fig-1}).
Specifically, $\rho_r=0$ and $\rho_r=1$ correspond to the minimum and maximum transportation costs, respectively, whereas $\rho_r=1/2$ corresponds to the midpoint of the interval. More generally, $\rho_r<1/2$ (respectively, $\rho_r>1/2$) indicates that $\mathbb E|X-Y|^r$ is closer to the lower (respectively, upper) transportation functional.
\end{remark}

\subsection{Examples: Normal approximation of classical counting distributions}

In this subsection, we illustrate the quantile representations obtained in Theorems~\ref{Wasserstein distance} and \ref{Wasserstein distance-1} for several classical counting distributions. Let
\[
S_n\equiv\sum_{i=1}^{n}X_i,
\]
where $X_1,\ldots,X_n$ are identically distributed nonnegative integer-valued random variables with finite mean $\mu$ and variance $\sigma^2>0$. Define the standardized sum
\[
Z_n\equiv\frac{S_n-n\mu}{\sigma\sqrt{n}},
\]
and let $\Phi$ denote the standard normal distribution function.

Since $Z_n$ is discrete, its quantile function is piecewise constant. Let
\[
p_{n,k}\equiv\mathbb P(S_n=k)
\quad \text{and} \quad 
F_n(k)\equiv\mathbb P(S_n\leqslant k),
\]
with the convention $F_n(-1)=0$. Then
\[
F_{Z_n}^{-1}(u)
=
\frac{k-n\mu}{\sigma\sqrt n} \,
\mathds{1}_{\{
F_n(k-1)<u\leqslant F_n(k)\}
}.
\]

Consequently,
\begin{align}
\mathcal W_r^r(F_{Z_n},\Phi)
&=
\sum_{k}
\int_{F_n(k-1)}^{F_n(k)}
\left|
\frac{k-n\mu}{\sigma\sqrt n}
-\Phi^{-1}(u)
\right|^r
{\rm d}u,
\label{W-discrete}
\\[0.2cm]
\overline{\mathcal W}_r^{\,r}(F_{Z_n},\Phi)
&=
\sum_{k}
\int_{F_n(k-1)}^{F_n(k)}
\left|
\frac{k-n\mu}{\sigma\sqrt n}
-\Phi^{-1}(1-u)
\right|^r
{\rm d}u.
\label{Wbar-discrete}
\end{align}

Using the symmetry relation
$
\Phi^{-1}(1-u)=-\Phi^{-1}(u),
$
the latter may be written as
\[
\overline{\mathcal W}_r^{\,r}(F_{Z_n},\Phi)
=
\sum_{k}
\int_{F_n(k-1)}^{F_n(k)}
\left|
\frac{k-n\mu}{\sigma\sqrt n}
+\Phi^{-1}(u)
\right|^r
{\rm d}u.
\]

\begin{remark}\label{rem-pre}
Note that for $r=2$, the Wasserstein distance and the upper transportation functional are completely determined by the quantile correlation coefficient
\[
\gamma_2(F_{Z_n},\Phi)
\equiv 
\int_0^1
F_{Z_n}^{-1}(u)\Phi^{-1}(u){\rm d}u.
\]
Since both $Z_n$ and the standard normal distribution have zero mean and unit variance, the identities
\[
\mathcal W_2^2(F_{Z_n},\Phi)
=
2[1-\gamma_2(F_{Z_n},\Phi)]
\]
and
\[
\overline{\mathcal W}_2^{\,2}(F_{Z_n},\Phi)
=
2[1+\gamma_2(F_{Z_n},\Phi)]
\]
hold. Hence, normal approximation in quadratic Wasserstein distance is equivalent to the convergence of $\gamma_2(F_{Z_n},\Phi)$ to one.
\end{remark}

\begin{remark}
Under the assumptions of the classical central limit theorem and the existence of a finite absolute moment of order $r$, one has
$
\mathcal W_r(F_{Z_n},\Phi)\to 0.
$
Thus, the representations derived in this section provide explicit transportation-based measures of deviation from normality.
\end{remark}

Although the previous results do not require independence, the following examples assume that $X_1,\ldots,X_n$ are independent and identically distributed, which allows explicit computation of the distribution of $S_n$ and the associated transportation functionals.

\subsubsection{Poisson Distribution}

Suppose that
$
X_i\sim\mathrm{Poisson}(\lambda),
\ i=1,\ldots,n.
$
Then
$
S_n\sim\mathrm{Poisson}(n\lambda),
$
and
\[
Z_n
=
\frac{S_n-n\lambda}
{\sqrt{n\lambda}}.
\]

The support points of $Z_n$ are
$
z_{n,k}
=
({k-n\lambda})/
{\sqrt{n\lambda}},
\ k=0,1,\ldots,
$
and
\[
F_n(k)
=
\exp({-n\lambda})
\sum_{j=0}^{k}
\frac{(n\lambda)^j}{j!}.
\]

Substituting into \eqref{W-discrete} and \eqref{Wbar-discrete} yields
\begin{align*}
\mathcal W_r^r(F_{Z_n},\Phi)
&=
\sum_{k=0}^{\infty}
\int_{F_n(k-1)}^{F_n(k)}
\left|
\frac{k-n\lambda}
{\sqrt{n\lambda}}
-\Phi^{-1}(u)
\right|^r
{\rm d}u,
\\[0.2cm]
\overline{\mathcal W}_r^{\,r}(F_{Z_n},\Phi)
&=
\sum_{k=0}^{\infty}
\int_{F_n(k-1)}^{F_n(k)}
\left|
\frac{k-n\lambda}
{\sqrt{n\lambda}}
+\Phi^{-1}(u)
\right|^r
{\rm d}u.
\end{align*}

\subsubsection{Binomial Distribution}

Let
$
X_i\sim\mathrm{Bernoulli}(p).
$
Then
$
S_n\sim\mathrm{Binomial}(n,p),
$
and
\[
Z_n
=
\frac{S_n-np}
{\sqrt{np(1-p)}}.
\]

The support points are
$
z_{n,k}
=
({k-np})/
{\sqrt{np(1-p)}},
\ k=0,\ldots,n,
$
and
\[
F_n(k)
=
\sum_{j=0}^{k}
\binom{n}{j}
p^j(1-p)^{n-j}.
\]

Therefore,
\begin{align*}
\mathcal W_r^r(F_{Z_n},\Phi)
&=
\sum_{k=0}^{n}
\int_{F_n(k-1)}^{F_n(k)}
\left|\frac{k-np}
{\sqrt{np(1-p)}}-\Phi^{-1}(u)\right|^r
{\rm d}u,
\\[0.2cm]
\overline{\mathcal W}_r^{\,r}(F_{Z_n},\Phi)
&=
\sum_{k=0}^{n}
\int_{F_n(k-1)}^{F_n(k)}
\left|\frac{k-np}
{\sqrt{np(1-p)}}+\Phi^{-1}(u)\right|^r
{\rm d}u.
\end{align*}

\subsubsection{Negative Binomial Distribution}

Assume that
$
X_i\sim\mathrm{NB}(r_0,p),
$
where
$
\mu={r_0(1-p)}/{p},
\
\sigma^2={r_0(1-p)}/{p^2}.
$
Then
$
S_n
\sim
\mathrm{NB}(nr_0,p),
$
and
\[
Z_n
=
\frac{
S_n-\frac{nr_0(1-p)}{p}
}
{
\sqrt{{nr_0(1-p)}/{p^2}}
}.
\]

The support points are
$
z_{n,k}
=
({
k-{nr_0(1-p)}/{p}
})/
{
\sqrt{{nr_0(1-p)}/{p^2}}
},
\ k=0,1,\ldots,
$
while
\[
F_n(k)
=
\sum_{j=0}^{k}
\binom{nr_0+j-1}{j}
p^{nr_0}(1-p)^j.
\]

Hence,
\begin{align*}
\mathcal W_r^r(F_{Z_n},\Phi)
&=
\sum_{k=0}^{\infty}
\int_{F_n(k-1)}^{F_n(k)}
\left|\frac{
	k-\frac{nr_0(1-p)}{p}
}
{
	\sqrt{{nr_0(1-p)}/{p^2}}
}-\Phi^{-1}(u)\right|^r
{\rm d}u,
\\[0.2cm]
\overline{\mathcal W}_r^{\,r}(F_{Z_n},\Phi)
&=
\sum_{k=0}^{\infty}
\int_{F_n(k-1)}^{F_n(k)}
\left|\frac{
	k-\frac{nr_0(1-p)}{p}
}
{
	\sqrt{{nr_0(1-p)}/{p^2}}
}+\Phi^{-1}(u)\right|^r
{\rm d}u.
\end{align*}

These representations express both the Wasserstein distance and the upper transportation functional entirely in terms of the distribution function of the underlying counting model. Consequently, they provide explicit measures of normal approximation that can be evaluated numerically without solving any transportation problem.


\section{Concluding remarks}\label{Concluding remarks}

This paper developed a unified probabilistic framework based on distributional derivatives and Dirac delta representations for deriving and interpreting a variety of statistical quantities. The proposed approach provides a common treatment of absolutely continuous, discrete, and mixed random variables, yielding unified formulas for conditional expectations, conditional distributions, hazard functions, and improper distributions. These results reveal a common localization mechanism underlying conditioning operations and related probabilistic constructions.

The same framework was combined with copula methods to investigate transportation and dispersion functionals through dependence structures. By exploiting the extremal properties of the Fr\'echet--Hoeffding bounds and the ordering of expectations induced by $\Delta$-antitonic functions, sharp bounds were established for moments of the form $\mathbb E|X-Y|^r$ under fixed marginals. These bounds admit equivalent representations in terms of marginal distribution functions, quantile functions, and survival functions, thereby providing complementary perspectives on transportation-type quantities. As consequences, concise derivations were obtained for the classical quantile representation of the Wasserstein distance and for a corresponding upper transportation functional.

The resulting representations also yield explicit bounds for generalized absolute difference moments expressed solely in terms of marginal survival functions. In the particular case $r=1$, these bounds provide survival-function representations for the bivariate Gini mean difference and the associated bivariate Gini index introduced by Capaldo and Navarro \cite{Capaldo2025}. From this perspective, these dependence-based measures of dispersion can be naturally interpreted within the same transportation framework developed throughout the paper. Moreover, the Gini mean difference emerges as an extremal first-order transportation cost, further strengthening the connection between inequality measures and optimal transport theory.

The quantile representations derived in this work further lead to explicit formulas for Wasserstein-type functionals associated with the normal approximation of standardized counting distributions. In particular, closed-form representations were obtained for Poisson, Binomial, and Negative Binomial models, providing computable measures of deviation from normality expressed entirely in terms of the corresponding distribution functions. For the quadratic case, the Wasserstein distance and the upper transportation functional were shown to be completely determined by a quantile correlation coefficient linking the standardized distribution to the standard normal law.

Overall, the results establish explicit connections among conditional structures, dependence modeling, dispersion measures, normal approximation, and optimal transportation within a unified probabilistic framework. Beyond providing unified derivations of several classical formulas, the proposed approach highlights the central role played by distributional representations and extremal dependence structures in a broad range of probabilistic and statistical problems. Future research directions include extensions to multivariate settings, stochastic processes, more general transportation costs, and asymptotic properties of the transportation-based and dependence-based measures studied in this paper.


\subsubsection*{Acknowledgements}
The research was supported in part by CNPq and CAPES grants from the Brazilian government.

\subsubsection*{Disclosure statement}
There are no conflicts of interest to disclose.

\subsubsection*{Data availability statement}
No data were generated or analyzed during this study.

\subsubsection*{Author contribution statement}
R.V., E.N., C.C.Y.D.: Conceptualization, Methodology, Formal analysis, Writing-original draft.

\end{document}